\documentclass[a4paper]{article}
\usepackage{latexsym,amsfonts,amssymb,amsmath,amsthm}

\newtheorem{theorem}{Theorem}[section]
\newtheorem{lemma}{Lemma}[section]
\newtheorem{proposition}{Proposition}[section]

\numberwithin{equation}{section}

\numberwithin{definition}{section}

\def\be{\begin{equation}}
\def\ee{\end{equation}}

\newcommand{\RR}{\ensuremath{\mathbb{R}}}
\newcommand{\QQ}{\ensuremath{\mathbb{Q}}}

\newcommand{\NN}{\ensuremath{\mathbb{N}}}
\newcommand{\PP}{\ensuremath{\mathbb{P}}}

\newcommand{\OFP}{(\Omega,\mathcal{F},\PP)}

\newcommand{\norm}[1]{\ensuremath{\lVert#1\rVert}}

\DeclareMathOperator{\spanned}{span}

\begin{document}

\title{Formulas for Generalized Principal Lyapunov Exponent for Parabolic PDEs\footnote{This is a pre-copy-editing, author-produced PDF of an article accepted for publication in \emph{Discrete and Continuous Dynamical Systems Series S} following peer review. The definitive publisher-authenticated version \emph{Discrete and Continuous Dynamical Systems Series S} \textbf{9}(4) (2016), pp. 1189--1199, is available online at:  http://dx.doi.org/10.3934/dcdss.2016048}}

\author{Janusz Mierczy\'nski
\thanks{Supported by project S40036/K1101.}
\\
Institute of Mathematics\\
Wroc{\l}aw University of Technology\\
Wybrze\.ze Wyspia\'nskiego 27\\
PL-50-370 Wroc{\l}aw\\
Poland\\
\\
and
\\
\\
Wenxian Shen
\\
Department of Mathematics\\
Auburn University\\
Auburn University, AL 36849\\
USA
}
\date{}

\maketitle

\begin{abstract}
An integral formula is given representing the generalized principal Lyapunov exponent for random linear parabolic PDEs.  As an application, an upper estimate of the exponent is obtained.
\end{abstract}



\section{Introduction}

In~\cite{MiShPart3} the current authors presented the theory of generalized principal Lyapunov exponents for linear random parabolic partial differential equations (PDEs) of the form
\begin{equation*}
\begin{aligned}
\frac{\partial u}{\partial t} = & \sum_{i=1}^{N} \frac{\partial}{\partial x_i} \biggl( \sum_{j=1}^{N} a_{ij}(\theta_{t}\omega,x) \frac{\partial u}{\partial x_{j}} + a_{i}(\theta_{t}\omega,x) u \biggr) +
\sum_{i=1}^N b_i(\theta_{t}\omega,x)\frac{\partial u}{\partial x_i} &
\\
& + c_0(\theta_{t}\omega,x)u, \qquad t > s, \ x \in D,
\end{aligned}
\end{equation*}
where $D \subset \mathbb{R}^N$, endowed with boundary conditions of either Dirichlet or Robin type, driven by an ergodic flow $(\theta_t)_{t \in \mathbb{R}}$ on a probability space $\OFP$.  The second- and first-order coefficients of the equation are assumed to be bounded uniformly in $\omega \in \Omega$, whereas as concerns zero-order coefficients some integral-type conditions are required.  That generalizes the theory in~\cite{MiSh3}, where it was assumed that all the coefficients are bounded.

The generalized principal Lyapunov exponent is defined as the logarithmic growth rate of some distinguished solutions (see Theorem~\ref{thm:floquet}(i)--(ii)).  It comes out that the generalized principal Lyapunov exponent is just equal to the top Lyapunov exponent (Theorem~\ref{thm:floquet}(iv)).  Moreover, for any nontrivial nonnegative solution its logarithmic growth rate equals the generalized principal Lyapunov exponent (Theorem~\ref{thm:floquet}(iii)).

Bearing in mind that Lyapunov exponents have relevance for establishing stability/instability of nonlinear PDEs of parabolic type, it is very important to find ways of estimating them.

It is the purpose of the present paper to give some estimates of formulas for the generalized principal Lyapunov exponent.

The paper is organized as follows.
In Section~\ref{sect:preliminary} some preliminaries are given.  In~particular, the standing assumptions are established, and necessary results from\break \cite{MiShPart3} are presented.
Section~\ref{sect:integral} is devoted to the proof of the integral formula of the generalized principal Lyapunov exponent (Theorem~\ref{thm:kappa}).
In Section~\ref{sect:estimates} an upper estimate is presented of the generalized principal Lyapunov exponent in terms of the principal eigenvalues of the (elliptic) equations with ``frozen'' coefficients (Theorem~\ref{thm:estimate}).

\section{Preliminaries}
\label{sect:preliminary}
In the present section we introduce the main concepts and assumptions, and formulate the main results on generalized principal Lyapunov exponents, as presented in~\cite{MiShPart3}.

By a {\em measurable space\/} we understand a pair $(P, \mathfrak{P})$, where $P$ is a set and $\mathfrak{P}$ is a $\sigma$\nobreakdash-\hspace{0pt}algebra of subsets of $P$.  For measurable spaces $(P, \mathfrak{P})$ and $(R, \mathfrak{R})$, a function $f \colon P \to R$ is {\em $(\mathfrak{P}, \mathfrak{R})$\nobreakdash-\hspace{0pt}measurable\/} if for any $A \in \mathfrak{R}$ its preimage, $f^{-1}(A)$, belongs to $\mathfrak{P}$.

By a {\em measure space\/} we understand a triple $(P, \mathfrak{P}, \mu)$, where $(P, \mathfrak{P})$ is a measurable space and $\mu$ is a measure defined on $\mathfrak{P}$.  When $\mu$ is finite we speak of a {\em finite measure space\/}, and when $\mu(P) = 1$ we speak of a {\em probability space\/}.

For a metrizable space $X$, $\mathfrak{B}(X)$ denotes the $\sigma$\nobreakdash-\hspace{0pt}algebra of Borel sets of $X$.

For further reference, we give now a special form of Pettis' theorem.
\begin{proposition}
\label{prop:Pettis}
Let $(P, \mathfrak{P}, \mu)$ be a finite measure space and let $C(Y)$ be the separable Banach space of all continuous real functions on a compact metrizable $Y$.  For a function $f \colon P \to C(Y)$ the following properties are equivalent.
\begin{itemize}
\item[\textup{(a)}]
$f$ is $(\mathfrak{P}, \mathfrak{B}(C(Y)))$\nobreakdash-\hspace{0pt}measurable.
\item[\textup{(b)}]
For each $y \in Y$ the function
\begin{equation*}
\bigl[ \, P \ni p \mapsto f(p)(y) \in \RR \, \bigr]
\end{equation*}
is $(\mathfrak{P}, \mathfrak{B}(\RR))$\nobreakdash-\hspace{0pt}measurable.
\end{itemize}
\end{proposition}
\begin{proof}
By \cite[Lemma 11.37 on p.\ 424]{AliB}, the function $f$ in (a) [resp.\ the functions $[\,  p \mapsto f(p)(y) \, ]$ in (b)] are measurable if and only if they are strongly measurable (that is, $\mu$\nobreakdash-\hspace{0pt}a.e.\ limits of simple functions).  \cite[Corollary on pp.\ 42--43]{DiUhl} implies that $f$ is strongly measurable if and only if for each $y \in Y$ the functions $[\,  p \mapsto f(p)(y) \, ]$ are strongly measurable.
\end{proof}

Assume that $(\OFP,(\theta_t)_{t \in \RR})$ is a {\em metric flow\/}:  $\OFP$ is a probability space, and $\theta \colon \RR \times \Omega \to \Omega$ is a $(\mathfrak{B}(\RR) \otimes \mathfrak{F},
\mathfrak{F})$\nobreakdash-\hspace{0pt}measurable mapping  such that the following holds, where, for $t \in \RR$ and $\omega \in \Omega$, $\theta_{t}\omega$ stands for $\theta(t,\omega)$:
\begin{itemize}
\item
$\theta_{0}\omega = \omega$ for any $\omega \in \Omega$,
\item
$\theta_{t_1 + t_2} \omega = \theta_{t_2}(\theta_{t_1}\omega)$ for
any $t_1, t_2 \in \RR$ and any $\omega \in \Omega$,
\item
for each $t \in \RR$ the mapping $\theta_t \colon \Omega \to \Omega$ is
$\PP$\nobreakdash-\hspace{0pt}preserving (i.e., $\PP(\theta_t^{-1}(F)) = \PP(F)$ for any $F \in \mathfrak{F}$ and $t \in \RR$).
\end{itemize}
Sometimes we write simply $(\theta_t)_{t \in \RR}$ for a metric flow.

Throughout the paper the standing assumption is that $(\OFP,(\theta_t)_{t \in \RR})$ is a metric flow which is moreover {\em ergodic\/}:  If $F \in \mathfrak{F}$ is such that $\theta_{t}(F) = F$ for all $t \in \RR$ (in other words, $F$ is {\em invariant\/}), then either $\PP(F) = 0$ or $\PP(F) = 1$. Furthermore, the probability measure $\PP$ is assumed to be complete.

Consider a family, indexed by $\omega \in \Omega$, of linear second order partial differential equations
\begin{equation}
\label{main-eq}
\begin{aligned}
\frac{\partial u}{\partial t} = & \sum_{i=1}^{N} \frac{\partial}{\partial x_i} \biggl( \sum_{j=1}^{N} a_{ij}(\theta_{t}\omega,x) \frac{\partial u}{\partial x_{j}} + a_{i}(\theta_{t}\omega,x) u \biggr) +
\sum_{i=1}^N b_i(\theta_{t}\omega,x)\frac{\partial u}{\partial x_i} &
\\
& + c_0(\theta_{t}\omega,x)u, \qquad t > s, \ x \in D,
\end{aligned}
\end{equation}
where $s \in \RR$ is an initial time and $D \subset \RR^N$ is a bounded domain with boundary $\partial D$, complemented with boundary condition
\begin{equation}
\label{main-bc}
\mathcal{B}_{\theta_{t}\omega} u = 0, \quad t > s, \ x \in \partial D,
\end{equation}
where
\begin{equation*}
\mathcal{B}_{\omega} u =
\begin{cases}
u & \text{(Dirichlet)}
\\[1.5ex]
\displaystyle \sum_{i=1}^N \biggl( \sum_{j=1}^N a_{ij}(\omega, x) \frac{\partial u}{\partial x_j} + a_{i}(\omega, x) u \biggr) \nu_i & \text{(Neumann).}
\\[1.5ex]
\displaystyle \sum_{i=1}^N \biggl( \sum_{j=1}^N a_{ij}(\omega, x) \frac{\partial u}{\partial x_j} + a_{i}(\omega, x) u \biggr) \nu_i + d_0(\omega, x)u  & \text{(Robin).}
\end{cases}
\end{equation*}
Above, $\nu = (\nu_1, \dots, \nu_{N})$ denotes the unit normal vector pointing out~of $\partial D$.

When we want to emphasize that \eqref{main-eq}+\eqref{main-bc} is considered for some (fixed) $\omega \in \Omega$ we write \eqref{main-eq}$_{\omega}$+\eqref{main-bc}$_{\omega}$.

Throughout the present paper, $\norm{\cdot}$ stands for the standard norm in $L_2(D)$ or for the standard norm in $\mathcal{L}(L_2(D))$ (= the Banach space of bounded linear operators from $L_2(D)$ into $L_2(D)$), depending on the context.  $L_2(D)^{+}$ denotes the set of those functions in $L_2(D)^{+}$ that are nonnegative Lebesgue-a.e.\ on $D$.

We make the following assumptions ($\alpha$ is a positive constant).

\medskip
\noindent \textbf{(A0)} (Boundary regularity) {\em $D \subset \RR^{N}$ is a bounded domain with boundary $\partial D$ of class $C^{3 + \alpha}$, for some $\alpha > 0$.}
\newline
(This is the first item in assumption (R)(ii) in~\cite{MiShPart3}.)

\medskip
\noindent \textbf{(A1)} (Measurability) {\em The functions $a_{ij} \colon \Omega \times D \to \RR$ \textup{(}$i, j = 1, \dots, N$\textup{)}, $a_{i} \colon \Omega \times D \to \RR$ \textup{(}$i = 1, \dots, N$\textup{)}, $b_{i} \colon \Omega \times D \to \RR$
\textup{(}$i = 1, \dots, N$\textup{)} and $c_{0} \colon \Omega \times D \to \RR$ are $(\mathfrak{F} \otimes \mathfrak{B}(D), \mathfrak{B}(\RR))$-measurable.  In the case of Robin boundary conditions the function $d_{0} \colon \Omega \times \partial D \to [0, \infty)$ is $(\mathfrak{F} \otimes \mathfrak{B}(\partial D), \mathfrak{B}(\RR))$\nobreakdash-\hspace{0pt}measurable.}
\newline
(This is assumption (PA1) in~\cite{MiShPart3}.)

For $\omega \in \Omega$ let $a_{ij}^{\omega} \colon \RR \times D \to \RR$ be defined as
\begin{equation*}
a_{ij}^{\omega}(t, x) := a_{ij}(\theta_{t}\omega, x),
\end{equation*}
and similarly for $a_{i}^{\omega}$, etc.

\medskip
\noindent \textbf{(A2)}
\begin{enumerate}
\item[\textup{(i)}]
(Boundedness of second order terms) {\em For each $\omega \in \Omega$ the functions $a_{ij}^{\omega}$ \textup{(}$i, j = 1, \dots, N$\textup{)} and $a_{i}^{\omega}$ \textup{(}$i= 1, \dots, N$\textup{)} belong to $C^{2+\alpha,2+\alpha}(\RR \times \bar{D})$.  Moreover, their $C^{2+\alpha,2+\alpha}(\RR \times \bar{D})$\nobreakdash-\hspace{0pt}norms are bounded uniformly in $\omega \in \Omega$.}
\item[\textup{(ii)}]
(Boundedness of first order terms) {\em For each $\omega \in \Omega$ the functions $b_{i}^{\omega}$ \textup{(}$i= 1, \dots, N$\textup{)} belong to $C^{2+\alpha,1+\alpha}(\RR \times \bar{D})$.  Moreover, their
$C^{2+\alpha,1+\alpha}(\RR \times \bar{D})$\nobreakdash-\hspace{0pt}norms are bounded uniformly in $\omega \in \Omega$.}
\item[\textup{(iii)}]
(Boundedness of boundary terms) {\em In the Robin boundary condition case, for each  $\omega \in \Omega$ the function $d_{0}^{\omega}$ belongs to $C^{2+\alpha,2+\alpha}(\RR \times \partial D)$.  Moreover, their $C^{2+\alpha,2+\alpha}(\RR \times \partial D)$\nobreakdash-\hspace{0pt}norms are bounded uniformly in $\omega \in \Omega$.}
\item[\textup{(iv)}]
(Local regularity of zero order terms) {\em
For each $\omega \in \Omega$ and each $T > 0$ the restriction $c^{\omega}_0|_{[0, T] \times \bar{D}}$ belongs to $C^{3, 2}([0, T] \times \bar{D})$;
}
\end{enumerate}
\noindent
((A2)(i)--(iii) are just the second, third and fourth items in assumption (R)(ii) in~\cite{MiShPart3}.)

\smallskip
For each $\omega \in \Omega$ put
\begin{equation*}
c_0^{(-)}(\omega) := - \bigl( \inf\limits_{x \in \bar{D}} c_0(\omega, x) \bigr)^{-},
\qquad
c_0^{(+)}(\omega) :=  \bigl( \sup\limits_{x \in \bar{D}} c_0(\omega, x) \bigr)^{+}.
\end{equation*}
Under (A2)(iv), for each $\omega \in \Omega$, $-\infty < c_0^{(-)}(\omega) \le 0 \le c_0^{(+)}(\omega) < \infty$.  Moreover, the following result holds, which we state here for further reference.
\begin{lemma}
\label{lm:c-zero-pm}
Assume \textup{(A1)} and \textup{(A2)(iv)}.
\begin{enumerate}
\item[\textup{(i)}]
The mappings $[ \, \Omega \ni \omega \mapsto c_0^{(\pm)}(\omega) \in \RR \, ]$
are $(\mathfrak{F}, \mathfrak{B}(\RR))$\nobreakdash-\hspace{0pt}measurable.
\item[\textup{(ii)}]
For each $\omega \in \Omega$ the mappings $[\,\RR \ni t \mapsto c_0^{(\pm)}(\theta_{t}\omega) \in \RR \,]$ are continuous.
\end{enumerate}
\end{lemma}
\begin{proof}
In order to prove part (i) it suffices to prove that $[ \,\omega \mapsto \inf\limits_{x \in \bar{D}} c_0(\omega, x) \, ]$ and $[ \, \omega \mapsto \sup\limits_{x \in \bar{D}} c_0(\omega, x) \, ]$ are $(\mathfrak{F}, \mathfrak{B}(\RR))$\nobreakdash-\hspace{0pt}measurable.  It follows from (A2)(iv) that $c_0(\omega, \cdot)$ belongs to $C(\bar{D})$.  Let $\{x_k\}_{k=1}^{\infty}$ be a dense subset of $\bar{D}$.  It is a consequence of (A1) that for each $k$ the function $[ \,\omega \mapsto c_0(\omega, x_k) \, ]$ is $(\mathfrak{F}, \mathfrak{B}(\RR))$\nobreakdash-\hspace{0pt}measurable.  Since
\begin{equation*}
\inf\limits_{x \in \bar{D}} c_0(\omega, x) = \inf\limits_{k \in \NN} c_0(\omega, x_k) \text{ and } \sup\limits_{x \in \bar{D}} c_0(\omega, x) = \sup\limits_{k \in \NN} c_0(\omega, x_k),
\end{equation*}
we conclude the proof of part (i) by using the fact that the infimum/supremum of a countable family of measurable functions is measurable.

As (A2)(iv) implies that for each $\omega \in \Omega$ the mapping $[\,\RR \ni t \mapsto c_0^{\omega}(t, \cdot) \in C(\bar{D}) \,]$ is continuous, part (ii) follows in a straightforward way.
\end{proof}

\medskip
\noindent \textbf{(A3)} (Ellipticity) {\em There exists $\alpha_0 > 0$ such that for each $\omega \in \Omega$ there holds
\begin{equation*}
\sum\limits_{i,j=1}^{N} a_{ij}(\theta_{t}\omega, x) \xi_{i} \xi_{j} \ge \alpha_0 \sum\limits_{i=1}^{N} \xi_{i}^2, \qquad \xi = (\xi_1, \dots, \xi_{N}) \in \RR^{N},
\end{equation*}
and
\begin{equation*}
a_{ij}(\theta_{t}\omega, x) = a_{ji}(\theta_{t}\omega, x), \qquad i, j = 1, 2, \dots, N,
\end{equation*}
for Lebesgue\nobreakdash-\hspace{0pt}a.e.\ $(t, x) \in \RR \times D$.
}
\newline (This is assumption (PA3) in~\cite{MiShPart3}.)

\medskip
Under (A0) through (A3), for any $\omega \in \Omega$, $s \in \RR$ and $u_0 \in L_2(D)$ there exists a unique {\em global weak solution\/}, $u_0(\cdot, s, \omega, u_0)$, of \textup{(\ref{main-eq})$_{\omega}$+(\ref{main-bc})$_{\omega}$} with initial condition $u(s) = u_0$.  Moreover, this weak solution is in~fact a {\em classical\/} solution:
\eqref{main-eq}$_{\omega}$ is satisfied pointwise on $(s, \infty) \times D$ and \eqref{main-bc}$_{\omega}$ is satisfied pointwise on $(s, \infty) \times \partial D$ (see, e.g., \cite[Prop. 2.5.1]{MiSh3}).

Define
\begin{equation}
\label{U-eq}
U_{\omega}(t)u_0 := u(t,0;\omega,u_0), \quad (\omega \in \Omega, \ t \ge 0, \ u_0 \in L_2(D)).
\end{equation}
$\Phi = ((U_\omega(t))_{\omega \in \Omega, t \in \RR^{+}}, (\theta_t)_{t\in\RR})$ is a measurable linear skew-\hspace{0pt}product semidynamical system on $L_2(D)$ covering a metric dynamical system $(\theta_{t})_{t \in \RR}$, called the {\em measurable linear skew-product semiflow\/} on $L_2(D)$ {\em generated by\/} \eqref{main-eq}+\eqref{main-bc} (for definitions see~\cite{MiShPart3}).

\begin{lemma}
\label{lm:gamma}
Assume \textup{(A0)--(A3)}.  Then there exists $\gamma \in \RR$ such that
\begin{equation*}
\norm{U_{\omega}(t) u_0} \le e^{{\gamma} t} \exp\biggl(\int\limits_{0}^{t} c_0^{(+)}(\theta_{\tau} \omega) \, d\tau \biggr) \norm{u_0}
\end{equation*}
for all $\omega \in \Omega$, $u_0 \in L_2(D)^{+}$ and $t > 0$.
\end{lemma}
\begin{proof}
The inequality follows by~\cite[Props.\ 3.4 and 3.5]{MiShPart3} and the monotonicity of the norm $\norm{\cdot}$ (see, e.g., \cite[Subsect.\ 2.3]{MiShPart1}).
\end{proof}

We make further assumptions on zero-order terms.

\smallskip
\noindent\textbf{(A4)} (Zero order terms) {\em
\begin{itemize}
\item[\textup{(i)}]
The mapping $\bigl[\, \Omega \ni \omega \mapsto \int\limits_0^1 c_0^{(+)}(\theta_{t}\omega) \, dt \in [0, \infty) \,\bigr]$ belongs to $L_1(\OFP)$;
\item[\textup{(ii)}]
the mapping $\bigl[\, \Omega \ni \omega \mapsto \ln^{+} \int\limits_0^1 \bigl(c_0^{(+)}(\theta_{t}\omega) - c_0^{(-)}(\theta_{t}\omega)\bigr) \, dt \in [0, \infty) \,\bigr]$ belongs to $L_1(\OFP)$;
\end{itemize}
}
\noindent (These are assumptions (PA0)(i)--(ii) in~\cite{MiShPart3}.)

\begin{lemma}
\label{lm:L1}
Assume \textup{(A0)--(A2)} and \textup{(A4)(i)}.  Then $c_0^{(+)} \in L_1(\OFP)$.
\end{lemma}
\begin{proof}
By Lemma~\ref{lm:c-zero-pm}, $[\, \Omega \ni \omega \mapsto c_0^{(+)}(\theta_{t}\omega) \,]$ is, for a fixed $t \in \RR$, $(\mathfrak{F}, \mathfrak{B}(\RR))$\nobreakdash-\hspace{0pt}measurable, and $[\, \RR \ni t \mapsto c_0^{(+)}(\theta_{t}\omega) \,]$ is, for a fixed $\omega \in \Omega$, continuous.  It follows from \cite[Lemma 4.51 on p.~153]{AliB} that  the function
\begin{equation*}
\bigl[\, \Omega \times [0,1] \ni (\omega, t) \mapsto c_0^{(+)}(\theta_{t}\omega) \in \RR \, \bigr]
\end{equation*}
is $(\mathfrak{F} \otimes \mathfrak{B}([0,1]), \mathfrak{B}(\RR))$\nobreakdash-\hspace{0pt}measurable.  Since $c_0^{(+)}$ is nonnegative, we can apply Tonelli's theorem to conclude that
\begin{equation*}
\begin{aligned}
\int\limits_{\Omega} \biggl( \int\limits_{0}^{1} c_0^{(+)}(\theta_{t}\omega) \, dt \biggr) \, d\PP(\omega) & = \int\limits_{0}^{1} \biggl( \int\limits_{\Omega}  c_0^{(+)}(\theta_{t}\omega) \, d\PP(\omega) \biggr) \, dt
\\
& = \int\limits_{0}^{1} \biggl( \int\limits_{\Omega}  c_0^{(+)}(\omega) \, d(\theta_{t}\PP)(\omega) \biggr) \, dt.
\end{aligned}
\end{equation*}
But $\theta_{t}\PP = \PP$ for any $t \in \RR$, so it follows that $c_0^{(+)} \in L_1(\OFP)$.
\end{proof}

We give now one of the main results formulated and proved in \cite{MiShPart3}.
\begin{theorem}
\label{thm:floquet}
Assume \textup{(A0)--(A4)}. Then there are:
\begin{itemize}
\item[$\bullet$]
an invariant set $\tilde{\Omega}_0 \subset \Omega$, $\PP(\tilde{\Omega}_0) = 1$,
\item[$\bullet$]
an $(\mathfrak{F},
\mathfrak{B}(L^2(D)))$\nobreakdash-\hspace{0pt}measurable function $w \colon \tilde{\Omega}_0 \to L_2(D)^+$ with $\norm{w(\omega)} = 1$ for all $\omega \in \tilde{\Omega}_0$,
\end{itemize}
having the following properties:
\begin{itemize}
\item[\textup{(i)}]
\begin{equation*}
w(\theta_{t}\omega) = \frac{U_{\omega}(t)w(\omega)} {\norm{U_{\omega}(t)w(\omega)}}
\end{equation*}
for any $\omega \in \tilde{\Omega}_0$ and $t \ge 0$.
\item[\textup{(ii)}]
There is $\tilde{\lambda}_1 \in [-\infty,\infty)$ such that for any $\omega\in\tilde \Omega_0$,
\begin{equation*}
\tilde{\lambda}_1 = \lim_{t\to \infty} \frac{\rho_t(\omega)}{t} =
\int_{\Omega} \ln\rho_1 \, d\PP,
\end{equation*}
where
\begin{equation*}
\rho_t(\omega) = \norm{U_{\omega}(t)w(\omega)} \quad \text{for } t \ge 0
\end{equation*}
\item[\textup{(iii)}]
For any $\omega\in\tilde\Omega_0$ and $u \in L_2(D)^+ \setminus \{0\}$,
\begin{equation*}
\lim_{t \to \infty} \frac{\ln{\norm{U_{\omega}(t)u}}}{t} = \tilde{\lambda}_1.
\end{equation*}
\item[\textup{(iv)}]
For any $\omega\in\tilde\Omega_0$ and $u \in L_2(D) \setminus \{0\}$
\begin{equation*}
\limsup_{t \to \infty} \frac{\ln{\norm{U_{\omega}(t)u}}}{t} \le \tilde{\lambda}_1
\end{equation*}
and
\begin{equation*}
\lim_{t \to \infty} \frac{\ln{\norm{U_{\omega}(t)}}}{t} = \tilde{\lambda}_1.
\end{equation*}
\end{itemize}
\end{theorem}
$\tilde{\lambda}_1$ is referred to as the {\em generalized principal Lyapunov exponent\/} of~$\Phi$ (or of~\eqref{main-eq}+\break\eqref{main-bc}).

\section{Integral formula for generalized principal Lyapunov exponent}
\label{sect:integral}
In this section we give a representation of the generalized Lyapunov exponent as the integral over $\Omega$ of some function connected with the Dirichlet form.

We assume that (A0) through (A4) are satisfied.  Recall that, in~particular, $(\OFP,(\theta_t)_{t \in \RR})$ is an ergodic metric flow.

Let $V$ denote the Banach space as in~\cite[Sect.\ 3]{MiShPart3} (in the Dirichlet or Neumann cases $V$ is a closed subspace of the Sobolev space $W^{1}_{2}(D)$).

For $\omega \in \Omega$ the {\em Dirichlet form\/} $B_{\omega} = B_{\omega}(\cdot,\cdot)$ is a  bilinear form on $V$ defined as (we use summation convention)
\begin{multline}
\label{measurable-bilinear-form-DC}
B_{\omega}(u,v) \\
:= \int_D \bigl( (a_{ij}(\omega,x)\partial_{j} u + a_i(\omega,x)u) \partial_{i}v - (b_i(\omega,x) \partial_{i} u + c_0(\omega,x) u) v \bigr) \,dx, \quad u,v \in V,
\end{multline}
in the Dirichlet and Neumann boundary condition cases, and
\begin{align}
\label{measurable-bilinear-form-RC}
B_{\omega}(u,v) &:= \int_D \bigl( (a_{ij}(\omega,x) \partial_{j} u + a_i(\omega,x) u) \partial_{i} v - (b_i(\omega,x) \partial_{i} u + c_0(\omega,x)u) v \bigr) \,dx
\nonumber \\
&\quad + \int_{\partial D} d_0(\omega,x) u v \,dH_{N-1}, \quad u,v \in V,
\end{align}
in the Robin boundary condition case ($H_{N-1}$ denotes $(N-1)$\nobreakdash-\hspace{0pt}dimensional Hausdorff measure, which is, under (A0), the same as surface measure).

From the fact that any solution is classical it follows that $w(\omega)$ in Theorem~\ref{thm:floquet} belongs to $C^{1}(\bar{D})$, hence the function $\kappa \colon \Omega \to \RR$,
\begin{equation*}
\kappa(\omega) := - B_{\omega}(w(\omega), w(\omega)),
\end{equation*}
is well defined.

The following is the main result of this section.
\begin{theorem}
\label{thm:kappa}
\begin{equation*}
\tilde{\lambda}_1 = \int_{\Omega} \kappa(\omega) \, d\PP(\omega).
\end{equation*}
\end{theorem}
In the case of bounded zero-order terms an analog of Theorem~\ref{thm:kappa} was proved in~\cite{MiSh3} (cf.\ \cite[Thm.\ 3.5.3]{MiSh3}).  For analogs of the formula for other types of equations, see the survey paper~\cite{Mi-NDS}.

Before giving the proof of Theorem~\ref{thm:kappa} we formulate and prove a couple of auxiliary results.

\begin{lemma}
\label{lm:w-and-kappa-continuity}
For any $\omega \in \tilde{\Omega}_0$ the mappings
\begin{equation*}
\bigl[ \, (-\infty,\infty) \ni t \mapsto w(\theta_{t}\omega) \in C^1(\bar{D}) \, \bigr] \quad \text{and} \quad \bigl[ \, (-\infty, \infty) \ni t \mapsto \kappa(\theta_{t}\omega) \in \RR \, \bigr]
\end{equation*}
are continuous.
\end{lemma}
\begin{proof}
Fix $\omega \in \tilde{\Omega}_0$.  By Theorem~\ref{thm:floquet}(i), we have
\begin{equation*}
w(\theta_{t}\omega) = \frac{U_{\theta_{-2}\omega}(t + 2) w(\theta_{-2}\omega)}{\norm{U_{\theta_{-2}\omega}(t + 2) w(\theta_{-2}\omega)}} \quad \text{for } -1 \le t \le 1.
\end{equation*}
Proceeding along the lines of the proof of  \cite[Proposition~2.5.1]{MiSh3} one obtains that the mapping $\bigl[\, [-1, 1] \ni t \mapsto U_{\theta_{-2}\omega}(t + 2) w(\theta_{-2}\omega) \in C^{2 +\alpha}(\bar{D}) \,\bigr]$ is continuous, so {\em a~fortiori\/} that mapping with $C^{2 +\alpha}(\bar{D})$ replaced by $C^{1}(\bar{D})$.  We have thus proved that the restriction of the  mapping $\bigl[ t \mapsto w(\theta_{t}\omega) \, \bigr]$ to $[-1, 1]$ is continuous.  Since $\omega \in \tilde{\Omega}_0$ is arbitrary, the mapping is continuous on its whole domain $(-\infty,\infty)$.

The continuity of the mapping $\bigl[ t \mapsto \kappa(\theta_{t}\omega) \, \bigr]$ is a consequence of the continuity of the first mapping and (A2).
\end{proof}

\begin{lemma}
\label{lm:Lyapunov-integral-kappa}
For any $\omega \in \tilde{\Omega}_0$ there holds
\begin{enumerate}
\item[\textup{(i)}]
\begin{equation}
\label{eq:kappa-1}
\left. \frac{d}{d\tau}\norm{U_{\omega}(\tau) w(\omega)} \right|_{\tau = t} = \kappa(\theta_{t}\omega) \, \norm{U_{\omega}(t) w(\omega)}
\end{equation}
for all $t \in \RR$;
\item[\textup{(ii)}]
\begin{equation}
\label{eq:kappa-2}
\ln{\norm{U_{\omega}(t) w(\omega)}} = \int\limits_{0}^{t} \kappa(\theta_{\tau}\omega) \, d\tau,
\end{equation}
for all $t > 0$.
\end{enumerate}
\end{lemma}
\begin{proof}
It follows from \cite[Proposition~2.1.4]{MiSh3} and the definition of $\kappa$ that
\begin{equation*}
\norm{U_{\omega}(t) w(\omega)}^2 - \norm{U_{\omega}(s) w(\omega)}^2 = - 2 \int\limits_{s}^{t} \kappa(\theta_{\tau}\omega) \norm{U_{\omega}(\tau) w(\omega)}^2 \, d\tau
\end{equation*}
for any $0 \le s < t$.  As, by Lemma~\ref{lm:w-and-kappa-continuity}, the integrand on the right\nobreakdash-\hspace{0pt}hand side above is continuous in $\tau$, the statement (i) follows by standard calculus (for a similar reasoning, see \cite[Lemma 3.5.3]{MiSh3}).  Part (ii) is straightforward.
\end{proof}

\begin{lemma}
\label{lm:measurable-C3-2}
For each $T > 0$ the mapping
\begin{equation*}
\left[\, \Omega \ni \omega \mapsto c^{\omega}_0|_{[0, T] \times \bar{D}} \in C^{3,2}([0, T] \times \bar{D}) \,\right]
\end{equation*}
is $(\mathfrak{F}, \mathfrak{B}(C^{3,2}([0, T] \times \bar{D})))$-measurable.
\end{lemma}
\begin{proof}[Indication of proof]
We give only the first step of the proof.  Namely, we prove that the mapping
\begin{equation*}
\left[\, \Omega \ni \omega \mapsto \frac{\partial c^{\omega}_0}{\partial t}\biggr|_{[0, T] \times \bar{D}} \in C([0, T] \times \bar{D}) \,\right]
\end{equation*}
is $(\mathfrak{F}, \mathfrak{B}(C([0, T] \times \bar{D})))$-measurable.  In view of Proposition~\ref{prop:Pettis} this is equivalent to showing that for $t \in [0,T]$ and $x \in \bar{D}$ fixed the mapping
\begin{equation*}
\left[\, \Omega \ni \omega \mapsto \frac{\partial c^{\omega}_0}{\partial t}(t,x)  \in \RR \,\right]
\end{equation*}
is $(\mathfrak{F}, \mathfrak{B}(\RR))$-measurable, which follows in turn from the fact that, for each $n \in \NN$,
\begin{equation*}
\left[\, \Omega \ni \omega \mapsto \frac{c^{\omega}_0(t+\frac{1}{n},x) - c^{\omega}_0(t,x)}{\frac{1}{n}} \in \RR \,\right]
\end{equation*}
is $(\mathfrak{F}, \mathfrak{B}(\RR))$-measurable and that $\frac{\partial c^{\omega}_0}{\partial t}(t,x) =\lim\limits_{n \to \infty} \frac{c^{\omega}_0(t+\frac{1}{n},x) - c^{\omega}_0(t,x)}{\frac{1}{n}}.$

We apply the above reasoning now to the derivatives of $c^{\omega}_0$ in $t$ and $x_i$, of suitable orders.
\end{proof}

\begin{lemma}
\label{lm:w-measurability-C1}
The mapping $w \colon \tilde{\Omega}_0 \to C^{1}(\bar{D})$ is $(\mathfrak{F}, \mathfrak{B}(C^{1}(\bar{D})))$-measurable.
\end{lemma}

\begin{proof}
Write $\tilde{\Omega}_0 = \bigcup\limits_{M > 0} \hat{\Omega}_{M}$, where
\begin{equation*}
\hat{\Omega}_{M} := \{\, \omega \in \tilde{\Omega}_0 : \lVert c^{\omega}_0|_{[-1, 0] \times \bar{D}} \rVert_{C^{3, 2}([-1, 0] \times \bar{D})} \le M \,\}.
\end{equation*}
In view of Lemma~\ref{lm:measurable-C3-2}, for each $M > 0$ the set $\hat{\Omega}_{M}$ belongs to $\mathfrak{F}$.  So it suffices to prove the measurability of $w$ restricted to $\hat{\Omega}_{M}$, for each $M > 0$. In order to do this, observe first that (A2) implies that the closure, $\hat{Y}_{M}$, of
\begin{equation*}
\{ \, (a_{ij}^{\omega}|_{[-1, 0] \times \bar{D}}, a_{i}^{\omega}|_{[-1, 0] \times \bar{D}}, b_{i}^{\omega}|_{[-1, 0] \times \bar{D}}, c_0^{\omega}|_{[-1, 0] \times \bar{D}}, d_0^{\omega}|_{[-1, 0] \times \bar{D}} ) : \omega \in \hat{\Omega}_{M} \, \}
\end{equation*}
in the topology of $C([-1, 0] \times \bar{D}, \RR^{N^2+2N+1}) \times C([-1, 0] \times \partial D, \RR)$ is, by the Ascoli--Arzel\`a theorem, a compact metrizable space (for Dirichlet or Neumann boundary conditions we put $d_0^{\omega}$ constantly equal to zero), consisting of functions whose $(C^{2+\alpha, 2+\alpha}([-1,0] \times \bar{D}, \RR^{N^2}) \times C^{2+\alpha, 1+\alpha}([-1,0] \times \bar{D}, \RR^{N+1}) \times C^{2+\alpha, 2+\alpha}([-1,0] \times \partial D, \RR))$-norms are uniformly bounded.  Define $\hat{\mathcal{E}}_{M} \colon \hat{\Omega}_{M} \to \hat{Y}_{M}$ as
\begin{equation*}
\hat{\mathcal{E}}_{M}(\omega) := (a_{ij}^{\omega}|_{[-1, 0] \times \bar{D}}, a_{i}^{\omega}|_{[-1, 0] \times \bar{D}}, b_{i}^{\omega}|_{[-1, 0] \times \bar{D}}, c_0^{\omega}|_{[-1, 0] \times \bar{D}}, d_0^{\omega}|_{[-1, 0] \times \bar{D}} ).
\end{equation*}
The $(\mathfrak{F}, \mathfrak{B}(\hat{Y}_{M}))$\nobreakdash-\hspace{0pt}measurability of $\hat{\mathcal{E}}_{M}$ is a consequence of Proposition~\ref{prop:Pettis}.

For $\hat{a} = (\hat{a}_{ij}, \hat{a}_{i}, \hat{b}_{i}, \hat{c}_0, \hat{d}_0) \in \hat{Y}_{M}$ and $u_0 \in L_2(D)$ denote by $\hat{u}(t; \hat{a}, u_0)$ the solution, on $[-1, 0]$, of \eqref{main-eq}+\eqref{main-bc}, but with $a_{ij}(\theta_{t}\omega, x)$ replaced with $\hat{a}_{ij}(t, x)$, etc., satisfying the initial condition $\hat{u}(-1; \hat{a}, u_0) = u_0$.  By an argument as in the proof of~\cite[Proposition 2.5.4]{MiSh3}, the mapping
\begin{equation*}
\bigl[\, \hat{Y}_{M} \times L_2(D) \ni (\hat{a}, u_0) \mapsto \hat{u}(0; \hat{a}, u_0) \in C^{1}(\bar{D}) \, \bigr]
\end{equation*}
is continuous.

The restriction $w|_{\hat{\Omega}_{M}}$ equals the composition of
\begin{equation*}
\bigl[\, \hat{\Omega}_{M} \ni \omega \mapsto (\hat{\mathcal{E}}_{M}(\omega), w(\theta_{-1}\omega)) \in \hat{Y}_{M} \times L_2(D) \, \bigr],
\end{equation*}
which is $(\mathfrak{F}, \mathfrak{B}(\hat{Y}_{M}) \otimes \mathfrak{B}(L_2(D)))$\nobreakdash-\hspace{0pt}measurable by construction and Theorem~\ref{thm:floquet}, and
\begin{equation*}
\Bigl[\, \hat{Y}_{M} \times (L_2(D)^{+} \setminus \{0\}) \ni (\hat{a}, u_0) \mapsto \frac{\hat{u}(0; \hat{a}, u_0)}{\norm{\hat{u}(0; \hat{a}, u_0)}} \in C^{1}(\bar{D}) \, \Bigr],
\end{equation*}
which is continuous, since it follows from the parabolic strong maximum principle that $\hat{u}(0; \hat{a}, u_0) \ne 0$ for $u_0 \in L_2(D)^{+} \setminus \{0\}$. It then follows that $w \colon \hat{\Omega}_M \to C^{1}(\bar{D})$ is $(\mathfrak{F}, \mathfrak{B}(C^{1}(\bar{D})))$-measurable.
\end{proof}

\begin{lemma}
\label{lm:kappa-measurability-C1}
$\kappa \colon \Omega \to \RR$ is $(\mathfrak{F}, \mathfrak{B}(\RR))$\nobreakdash-\hspace{0pt}measurable.
\end{lemma}
\begin{proof}
By Proposition~\ref{prop:Pettis} and (A2), the mapping $\bigl[\, \Omega \ni \omega \mapsto a^{\omega}_{ij}(0, \cdot) \in C(\bar{D}) \, \bigr]$ is $(\mathfrak{F}, \mathfrak{B}(C(\bar{D})))$\nobreakdash-\hspace{0pt}measurable.  As the mapping
\begin{equation*}
\Biggl[ C(\bar{D}) \times C^{1}(\bar{D}) \times C^{1}(\bar{D}) \ni (z, u, v) \mapsto \int\limits_{D} \bigl( z(x) \partial_{j} u(x) \bigr) \partial_{i}v(x) \, dx \in \RR \Biggr]
\end{equation*}
is continuous, it follows that
\begin{equation*}
\Biggl[ \tilde{\Omega}_{0} \ni \omega \mapsto \int\limits_{D} \bigl( a_{ij}(\omega,x) \partial_{j} w(\omega) \bigr) \partial_{i} w(\omega) \, dx \in \RR \Biggr]
\end{equation*}
is $(\mathfrak{F}, \mathfrak{B}(\RR))$\nobreakdash-\hspace{0pt}measurable.  The measurability of the remaining summands is proved in a similar way.
\end{proof}

\begin{lemma}
\label{lm:kappa-summability}
$\kappa^{+} \in L_1(\OFP)$.
\end{lemma}
\begin{proof}
Application of~\eqref{eq:kappa-2}, combined with Lemma~\ref{lm:gamma}, implies that, for some $\gamma \in \RR$,
\begin{equation*}
\int\limits_{0}^{t} \kappa(\theta_{\tau}\omega) \, d\tau \le \int\limits_{0}^{t} c_0^{(+)}(\theta_{\tau}\omega) \, d\tau + \gamma t
\end{equation*}
for all $\omega \in \tilde{\Omega}_0$ and $t > 0$.  By the continuity of $\bigl[ \, [0,\infty) \ni t \mapsto \kappa(\theta_{t}\omega) \,\bigr]$ (see Lemma~\ref{lm:w-and-kappa-continuity}) and $\bigl[\, [0,\infty) \ni t \mapsto c_0^{(+)}(\theta_{t}\omega) \,\bigr]$ (see Lemma~\ref{lm:c-zero-pm}(ii)), there holds
\begin{equation*}
\kappa(\omega) \le c_0^{(+)}(\omega) + \gamma
\end{equation*}
for $\PP$\nobreakdash-\hspace{0pt}a.e.\ $\omega \in \Omega$, which gives, via Lemma~\ref{lm:L1}, the desired result.
\end{proof}

\begin{proof}[Proof of Theorem~\ref{thm:kappa}]
Application of the Birkhoff Ergodic Theorem (see, e.g., \cite{Kre}) gives that for $\PP$\nobreakdash-\hspace{0pt}a.e.\ $\omega \in \Omega$
\begin{equation*}
\lim\limits_{t \to \infty} \frac{1}{t} \int\limits_{0}^{t} \kappa(\theta_{\tau}\omega) \, d\tau = \int\limits_{\Omega} \kappa(\cdot) \, d\PP(\cdot),
\end{equation*}
from which the conclusion follows immediately.
\end{proof}

We would like to emphasize here the differences between the theory presented in~\cite{MiSh3} and the theory of measurable positive skew-product flows as presented in~\cite{MiShPart3}.  Namely, in the former a random (or nonautonomous) linear parabolic PDE generates a \textbf{topological} linear skew-product semiflow  (the base $Y$, that is, the set of parameters, is a compact metrizable space).  The distinguished solutions (whose logarithmic growth rates are called {\em principal Lyapunov exponents\/}) correspond to a \textbf{continuous} function $w$ defined on the base $Y$.  That simplifies the proofs considerably compared with these in the present paper.

\section{Estimates from above}
\label{sect:estimates}
In the present section we consider symmetric problems of the form
\begin{equation}
\label{main-eq-symmetric}
\begin{cases}
\displaystyle \frac{\partial u}{\partial t} = \sum_{i=1}^{N} \frac{\partial}{\partial x_i} \biggl( \sum_{j=1}^{N} a_{ij}(\theta_{t}\omega,x) \frac{\partial u}{\partial x_{j}} \biggr) + c_0(\theta_{t}\omega,x)u, & \qquad t > s, \ x \in D,
\\[1ex]
\mathcal{B}_{\theta_{t}\omega}u = 0 & \qquad t > s, \ x \in \partial D,
\end{cases}
\end{equation}
where
\begin{equation*}
\mathcal{B}_{\omega} u =
\begin{cases}
u & \text{(Dirichlet)}
\\[1.5ex]
\displaystyle \sum_{i=1}^N \biggl( \sum_{j=1}^N a_{ij}(\omega, x) \frac{\partial u}{\partial x_j} \biggr) \nu_i  & \text{(Neumann)}
\\[1.5ex]
\displaystyle \sum_{i=1}^N \biggl( \sum_{j=1}^N a_{ij}(\omega, x) \frac{\partial u}{\partial x_j} \biggr) \nu_i + d_0(\omega, x)u  & \text{(Robin).}
\end{cases}
\end{equation*}
To emphasize that \eqref{main-eq-symmetric} is considered for some (fixed) $\omega \in \Omega$ we write \eqref{main-eq-symmetric}$_{\omega}$.

We assume (A0) through (A4).

The Dirichlet form $B_{\omega}(\cdot,\cdot)$ takes the form:
\begin{equation}
\label{measurable-bilinear-form-DC-symmetric}
B_{\omega}(u,v) = \int_D ( a_{ij}(\omega,x)\partial_{j} u \, \partial_{i} v + c_0(\omega,x) u v ) \,dx, \quad u,v \in V,
\end{equation}
in the Dirichlet and Neumann cases, and
\begin{equation}
\label{measurable-bilinear-form-RC-symmetric}
B_{\omega}(u,v) = \int_D ( a_{ij}(\omega,x)\partial_{j} u \,  \partial_{i} v + c_0(\omega,x) u v ) \,dx + \int_{\partial D} d_0(\omega,x) u v \,dH_{N-1}, \quad u,v \in V,
\end{equation}
in the Robin case.

It is well known (see, e.g., \cite{Evans}) that, for fixed $\omega \in \Omega$, the largest (necessarily real) eigenvalue of the (elliptic) boundary value problem
\begin{equation}
\label{principal-symmetric}
\begin{cases}
\displaystyle \lambda u = \sum_{i=1}^{N} \frac{\partial}{\partial x_i} \biggl( \sum_{j=1}^{N} a_{ij}(\omega,x) \frac{\partial u}{\partial x_{j}} \biggr) + c_0(\omega,x)u, & \qquad x \in D,
\\[1ex]
\mathcal{B}_{\omega}u = 0 & \qquad x \in \partial D
\end{cases}
\end{equation}
equals
\begin{equation}
\label{eq:rayleigh}
\max\{\, -B_{\omega}(u,u) : u \in V, \norm{u} = 1 \,\}.
\end{equation}
We will denote this quantity (called the {\em principal eigenvalue of\/} \eqref{principal-symmetric}$_{\omega}$) by $\lambda_{\mathrm{princ}}(\omega)$.  Moreover, there exists $v \in V$, $\norm{v} = 1$, such that $v(x) > 0$ for all $x \in D$ and
the maximum in~\eqref{eq:rayleigh} is attained precisely at $v$ and $-v$.  Such a $v$ is called the {\em normalized principal eigenfunction of\/} \eqref{principal-symmetric}$_{\omega}$.

\begin{lemma}
\label{lm:principal-measurability-symmetric}
$\bigl[\, \Omega \ni \omega \mapsto \lambda_{\mathrm{princ}}(\omega) \in \RR \,\bigr]$ is $(\mathfrak{F},  \mathfrak{B}(\RR))$\nobreakdash-\hspace{0pt}measurable.
\end{lemma}
\begin{proof}
Since $V$ is separable, we can take a countable set $\{\, u_k \in V: k \in \NN \,\}$ such that $\norm{u_k} = 1$ for each $k \in \NN$ and $\spanned_{\QQ}\{\, u_k: k \in \NN \,\}$ is dense in $V$.  As $\bigl[\, V \setminus \{0\} \ni u \mapsto -B_{\omega}(u, u)/\norm{u}^2 \in \RR \,\bigr]$ is, for each $\omega \in \Omega$, continuous, we have that
\begin{equation}
\begin{aligned}
\lambda_{\mathrm{princ}}(\omega) & = \max{\biggl\{\, \frac{-B_{\omega}(u,u)}{\norm{u}^2} : u \in V, u \ne 0 \,\biggr\}}
\\
& = \sup{\biggl\{\, \frac{-B_{\omega}(u,u)}{\norm{u}^2} : u \in \spanned_{\QQ}\{\, u_k: k \in \NN \,\}, u \ne 0 \,\biggr\}}.
\end{aligned}
\end{equation}
In view of the above it suffices to prove, repeating reasoning as in the proof of Lemma~\ref{lm:kappa-measurability-C1}, that for each nonzero $u \in \spanned_{\QQ}\{\, u_k: k \in \NN \,\}$ the mapping $\bigl[\, \Omega \ni \omega \mapsto -B_{\omega}(u, u)/\norm{u}^2 \in \RR \,\bigr]$ is $(\mathfrak{F},  \mathfrak{B}(\RR))$\nobreakdash-\hspace{0pt}measurable.
\end{proof}

\begin{lemma}
\label{lm:principal-summability-symmetric}
$\bigl[\, \Omega \ni \omega \mapsto \lambda^{+}_{\mathrm{princ}}(\omega) \in \RR \,\bigr] \in L_1(\OFP)$.
\end{lemma}
\begin{proof}
By copying the reasoning as in the proof of Lemma~\ref{lm:kappa-summability} we obtain that
\begin{equation*}
\lambda_{\mathrm{princ}}(\omega) \le c_0^{(+)}(\omega) + \gamma
\end{equation*}
for all $\omega \in \Omega$, which gives, via Lemma~\ref{lm:L1}, the desired result.
\end{proof}

Since, by~\eqref{eq:rayleigh}, $\kappa(\omega) \le \lambda_{\mathrm{princ}}(\omega)$ for $\PP$\nobreakdash-\hspace{0pt}a.e.\ $\omega \in \Omega$, we have, in view of the ergodicity of the flow $(\theta_t)_{t \in \RR}$, the following result.
\begin{theorem}
\label{thm:estimate}
\begin{equation}
\tilde{\lambda}_1 \le \int_{\Omega} \lambda_{\mathrm{princ}}(\omega) \, d\PP(\omega).
\end{equation}
\end{theorem}

\section*{Acknowledgments}
The authors wish to thank the referees for their remarks.


\end{document}